\documentclass[12pt]{article}

\usepackage{amsmath}
\usepackage{amssymb}
\usepackage{amsthm}
\usepackage{amsfonts}
\usepackage{amscd}
\usepackage{url}
\usepackage{graphicx}
\usepackage{caption}
\usepackage{subcaption}
\usepackage[numbers,sort]{natbib}
\usepackage[usenames,dvipsnames]{color}
\usepackage{hyperref}

\usepackage{todonotes}

\hypersetup{
    colorlinks=true,
    breaklinks=true,
    filecolor=blue,
    urlcolor=Violet,
    linkcolor=blue,
    citecolor=Green,
    anchorcolor=blue,
    frenchlinks=true,
    pdfborder={0 0 0},
    pdfpagelayout=TwoPageRight,
    pdfdisplaydoctitle=true,
    bookmarksnumbered=true
}

\theoremstyle{plain}
\newtheorem{lem}{Lemma}
\newtheorem{prop}{Proposition}
\newtheorem{thm}{Theorem}

\theoremstyle{definition}
\newtheorem{rem}{Remark}

\renewcommand{\P}{\mathbb P}

\newcommand{\Z}{\mathbb Z}

\newcommand{\D}{\mathbb D}
\newcommand{\R}{\mathbb R}
\renewcommand{\S}{\mathbb S}
\renewcommand{\H}{\mathbb H}

\newcommand{\eps}{\epsilon}
\newcommand{\fig}[3]{{\includegraphics[height=#1cm,width=#2cm]{#3}}}
\newcommand{\U}{\mathcal U}

\DeclareMathOperator{\orb}{orb}
\DeclareMathOperator{\norm}{N}

\newcommand{\qi}{\mathbf{i}}
\newcommand{\qj}{\mathbf{j}}
\newcommand{\qk}{\mathbf{k}}

\newcommand{\e}{\mathfrak{e}}
\renewcommand{\o}{\mathfrak{o}}
\newcommand{\Mend}{\hfill \ensuremath{\vartriangleleft}}

\title{\vspace{-1.1cm}
Kinematic generation of Darboux cyclides}

\author{Niels Lubbes, Josef Schicho}

\begin{document}

\maketitle

\begin{abstract}
We state a relation between two families of lines that cover a quadric surface in 
the Study quadric and two families of circles that cover a Darboux cyclide.

\end{abstract}

\section{Introduction}

The Study quadric is a projective compactification of the group of Euclidean displacements.
If we fix a point in 3-space, then projective varieties inside the Study quadric --- considered
as sets of displacements --- give rise to orbit varieties in 3-space. 

Let us consider the relation between a class
of varieties in the Study quadric and their orbits in 3-space.
A classical example is presented by the class of lines in the Study quadric;
the orbit of a line is either a circle, line or point in 3-space (see Lemma~\ref{l}).
The orbit of a conic in the Study quadric is a rational quartic curve with full cyclicity \cite{or}.
The case of rational curves of arbitrary degree in the Study quadric is studied in \cite{bj}:
a general rational curve of degree $d$ in the Study quadric has an orbit of degree $2d$.
We also have a uniqueness result: 
for any rational curve of degree $d$ and cyclicity $2c$,
there is a unique rational curve in the Study quadric of degree $d-c$
in the Study quadric defining that orbit \citep[Theorem~2]{zjs}.

In this paper we show that the orbit of a doubly ruled quadric surface in the Study quadric is a Darboux cyclide.
Darboux cyclides are surfaces that contain at least two and at most six circles through each point. 
These surfaces have been recently studied in \cite{s1,s3,kz,p}.
Theorem~\ref{main} is again a uniqueness result:
for two families of circles that cover a Darboux cyclide there exists
a unique doubly ruled quadric surface in the Study quadric.

%

\section{The orbit map}

The \emph{dual quaternions} are defined as the noncommutative associative algebra
\[
\D\H:=
\R[\qi,\qj,\qk,\eps]
/
\langle \qi^2+1,\qj^2+1, \qk^2+1, \qi\qj\qk+1,\eps^2,\eps\qi-\qi\eps,\eps\qj-\qj\eps,\eps\qk-\qk\eps \rangle
.
\]
We consider the following coordinates for $h\in\D\H$ and $\overline{h}\in\D\H$:
\begin{gather*}
\label{eqn:pq}
h
=
p+q\eps
= 
(p_0+p_1\qi+p_2\qj+p_3\qk) + (q_4 + q_5\qi + q_6\qj + q_7\qk)\eps,
\\
\overline{h}=\overline{p}+\overline{q}\eps
= 
(p_0-p_1\qi-p_2\qj-p_3\qk) + (q_4 - q_5\qi - q_6\qj - q_7\qk)\eps.
\end{gather*}
We denote by $N:\D\H\to\D,~h\mapsto h\bar{h}$, the dual quaternion norm.
By projectivizing $\D\H$ as a real 8-dimensional vector space, we obtain $\P^7$. 
The \emph{Study quadric} is defined as 
\[
S:=\{~ h\in\P^7 ~|~ h\overline{h} \in \R ~\}
=
\{~ p+q\eps\in\P^7 ~|~ p_0q_0+p_1q_1+p_2q_2+p_3q_3=0 ~\}.
\]
The \emph{Study boundary} $B\subset S$ is defined as
$B:=\{~ h\in S ~|~ h\overline{h}=0 ~\}$.
If we identify $\R^3$ with $\{~ v\in \D\H ~|~ v=v_1\qi+v_2\qj+v_3\qk ~\}$, 
then the \emph{Study kinematic mapping} is a group action
\[
\varphi\colon (S\setminus B)\times \R^3 \to \R^3,\quad  (p+q\eps, v) \mapsto \frac{pv\bar{p}+p\bar{q}-q\bar{p}}{p\bar{p}},
\]
and $S\setminus B\cong SE(3)$ via this action \citep[Section 2.1]{hs}. 
We choose the following coordinates for the 3-dimensional \emph{M\"obius quadric}: 
\[
\S^3:=\{~ x\in\P^4 ~|~ x_0x_4-x_1^2-x_2^2-x_3^2=0 ~\}. 
\]
With this somewhat unusual choice of coordinates the \emph{stereographic projection} with center  
$(0:0:0:0:1)\in\S^3$ is defined as 
\[
\tau\colon \S^3\to \P^3,\quad (x_0:\ldots:x_4)\mapsto (x_0:x_1:x_2:x_3).  
\]
For any point $u=(u_0:\ldots:u_4)\in\S^3$ such that $u_0\neq0$, the \emph{orbit map} is defined as 
\begin{equation*}
\begin{array}{rrcl}
\orb_u\colon & S\setminus F_u       & \to & \S^3 \\
             & p+q\eps              & \mapsto & \left(p\bar{p}:w_1:w_2:w_3:4q\bar{q}-p\bar{p}v^2+2(qv\bar{p}-pv\bar{q})\right),
\end{array}
\end{equation*}
where 
$pv\bar{p}+p\bar{q}-q\bar{p}=w_1\qi+w_2\qj+w_3\qk$ with $v=\frac{u_1}{u_0}\qi+\frac{u_2}{u_0}\qj+\frac{u_3}{u_0}\qk$ the dehomogenization of $\tau(u)$ and 
\[
F_u:=\{~ p+q\eps \in S ~|~ p\bar{p}=pv\bar{p}+p\bar{q}-q\bar{p}=4q\bar{q}+2(qv\bar{p}-pv\bar{q})=0  ~\}\subset B.
\]
Notice that the orbit map is the composition of the projective closure of 
$\varphi(\cdot,v):S\to\R^3$ with the inverse stereographic projection.
We fix notation for the identity $\e:=(1:0:\ldots:0)\in S$ and the origin $\o:=(1:0:\ldots:0)\in \S^3$.

\begin{prop}
\label{orb}
The Zariski closure of the image of $\orb_u$ is $\S^3$ and $F_u\subset S$ is a quartic 4-fold. 
\end{prop}
 
\begin{proof}
Suppose that $p+\eps q\in S$ so that $p\bar{q}+q\bar{p}=0$
and $w:=pv\bar{p}+p\bar{q}-q\bar{p}$.
We find that $w\bar{w}=-(pv\bar{p}+2p\bar{q})^2
=
p\bar{p}\cdot (4q\bar{q}-p\bar{p}v^2+2(qv\bar{p}-pv\bar{q}))$.
It follows that the image of $\orb_u$ is contained in $\S^3$, 
since $w\bar{w}=w_1^2+w_2^2+w_3^2$.
Recall that $S\setminus B$ is isomorphic to $SE(3)$ via the Study kinematic mapping,
and thus the image of $\orb_u$ is Zariski dense in $\S^3$.

An Euclidean isometry of $\R^3$ 
is realized by an automorphism of $S$ that preserves $B$.
Thus there exists $e\in S\setminus B$ such that $\orb_\o(e)=u$. 
Since the automorphisms of $S\setminus B$ and $\R^3$ are transitive,
we may assume without loss of generality that 
$e=\e$ and $u=\o$. 
The locus where $\orb_\o$ is not defined equals
\[
F_\o=\{~ p+q\eps \in\D\H ~|~ p\bar{q}=q\bar{p}=p\bar{p}=q\bar{q}=0 ~\}\subset S.
\]
We computed with a computer algebra system the Hilbert function of the ideal of $F_\o$ 
and find that $\dim F_\o=\deg F_\o=4$.
This concludes the proof of this proposition.
\end{proof}


A \emph{circle} is an irreducible conic in $\S^3$. 
Let
\begin{align*}
{\cal L}_\e &:=\{~ \ell\subset S ~|~ \ell \text{ is a line such that } \e\in \ell ~\},
\\
{\cal C}_\o &:=\{~ C\subset\S^3 ~|~ C \text{ is either a circle or a point such that } \o\in C ~\}.
\end{align*}

\begin{lem}
\label{l}
If $\ell\subset S$ is a line, then $\orb_u(\ell)$ is a circle or a point.
Moreover, the following map is almost everywhere one-to-one
\[
\psi\colon {\cal L}_\e \to  {\cal C}_\o,\quad \ell \mapsto \orb_\o(\ell).
\]
\end{lem}

\begin{proof}
Lines in the Study quadric correspond to either rotations or translations \citep[Section~2.5]{hs}
and orbits under these 1-parameter subgroups are circles, lines or points in $\R^3$
und thus via the stereographic projecion $\tau$ points or circles in $\S^3$. 

We can associate to a circle in the 4-dimensional set ${\cal C}_\o$,
a unique 1-parameter subgroup of rotations corresponding to a line in ${\cal L}_\e$. 
There is a 2-dimensional set of lines in ${\cal L}_\e$ such that the
rotational axis of these lines passes through $\o$ and the corresponding Lie circle is the point $\o$.
Thus $\psi$ is one-to-one except for a lower dimensional subset as was claimed.
\end{proof}

\section{Quadric surfaces in the Study quadric}

A \emph{Darboux cyclide} is defined as a quartic weak del Pezzo surface in $\S^3$ \citep[Section~8.6.2]{dol}.
Such surfaces are the intersection of $\S^3$ with a quadric hypersurface \citep[Theorem~8.6.2]{dol}. 
Let $\U_\o$ denote the set of quadric surfaces $Q\subset S$ such that 
either there exists $V\cong\P^3$ such that
$Q\subset V\subset S$, or $Q\subset F_\o \subset S$.

\begin{lem}
\label{pmz}
If $Q\subset S$ is a doubly ruled quadric surface such that $Q\notin \U_\o$, 
then $\orb_\o(Q)\subset \S^3$ is a Darboux cyclide.
\end{lem}

\begin{proof}
There exists
bilinear homogeneous $a+\eps b \in \D\H[s_0,s_1,t_0,t_1]$ such that $Q$ is parametrized by
\[
\mu\colon \P^1\times\P^1\dashrightarrow Q\subset S,\quad (s_0:s_1;t_0:t_1)\mapsto a+\eps b,
\]
and $a\bar{b}-b\bar{a}=w_1\qi+ w_2\qj+ w_3\qk$ with $w_1,w_2,w_3\in\R[s_0,s_1,t_0,t_1]$ so that
\begin{equation*}
\begin{array}{rrcl}
\orb_\o\circ\mu\colon& \P^1\times\P^1 &\dashrightarrow& D\subset\S^3,
\\
&(s,t) &\mapsto& (a\bar{a}:w_1:w_2:w_3:4b\bar{b}). 
\end{array}
\end{equation*}
Thus the map $\orb_\o\circ\mu$ is of bidegree (2,2) into $\P^4$.
By Proposition~\ref{orb}, the map $\orb_\o$ is not defined at a quartic 4-fold $F_\o\subset S$.
We observe that $Q$ is the intersection of $S$ with a 3-space, since $Q\notin \U_\o$.
It follows that $Q\cap F_\o$ consists of 4 points (counted with multiplicity).
Thus $\orb_\o\circ\mu$ has 4 base points. A basis of all bidegree (2,2) functions on $\P^1\times\P^1$
defines a map whose image is a degree 8 weak del Pezzo surface $X\subset \P^8$.
A basis of bidegree (2,2) functions that pass through a basepoint, 
defines a map whose image is a projection of $X$ from a point so that the degree 
and embedding dimension drops by one.
Such a projection realizes the blowup of $X$ in a point and is again a weak del Pezzo surface 
\citep[Proposition~8.1.23]{dol}. 
Thus $M$ can be obtained as 4 subsequent projections of $X$, which 
results in a quartic weak del Pezzo surface in $\P^4$.
This concludes the proof of this lemma, since the image of $\orb_\o$ is $\S^3$ 
by Proposition~\ref{orb}. 
\end{proof}

\begin{figure}[b]
\centering
\fig{3}{4}{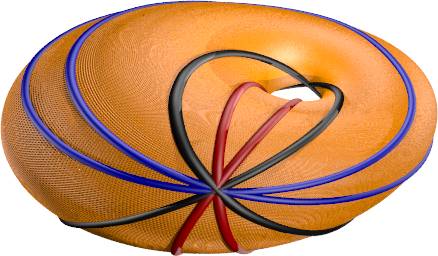}~~ 
\fig{3}{4}{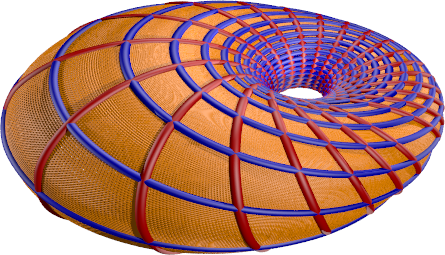}~~ 
\fig{3}{4}{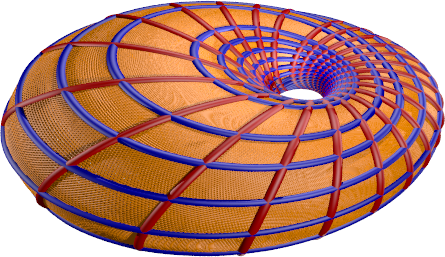} 
\caption{\it
A smooth Darboux cyclide contains six circles through each point
and admits $\binom{6}{2}-3$ pairs of families of circles $F$ and $F'$ so that $F\cdot F'=1$.}
\label{fig}
\end{figure}

A \emph{family of curves} of a surface $X$ is defined 
as an irreducible hypersurface $F\subset X\times\P^1$ 
such that the closure of the first projection of $F$ equals $X$.
A curve $F_t\subset X$ in the family $F$ for some $t\in \P^1$ is defined as $\pi_1( F\cap X\times\{t\})$.
If $F$ and $F'$ are families of $X$, then we denote by $F\cdot F'$ the number of intersections
of a general curve in $F$ and a general curve in $F'$. See also Figure~\ref{fig}.

\begin{lem}
\label{darboux}
If two different Darboux cyclides in $\S^3$ intersect in three circles,
then two of these circles are co-spherical.
\end{lem}

\begin{proof}
Suppose that $D,D'\subset \S^3$ are Darboux cyclides.
We can associate to the weak del Pezzo surface $D$ its Picard group, which is a quadratic lattice
$\langle \alpha_0,\alpha_1,\ldots,\alpha_5\rangle_\Z$ 
with intersection pairing $\alpha_0^2=1$, $\alpha_i^2=-1$ for $i>0$ and $\alpha_i\cdot\alpha_j=0$ for $i\neq j$ 
\citep[Section~8.2.1]{dol}. 
We associate to a curve $C\subset D$ its divisor class $[C]$ in the Picard group of $D$.
The class of any hyperplane section of $D$ is equal to the anticanonical class $-\kappa=3\alpha_0-\alpha_1-\ldots-\alpha_5$.
Both $D$ and $D'$ are intersections of $\S^3$ with a quadric hypersurface
so that $[D\cap D']=-2\kappa$ and $\deg D\cap D'=8$. 
Since $D\cap D'$ contains 3 circles by assumption, 
it follows that $D\cap D'$ consists of 4 conics.
The class of a conic in $M$ is either
$\alpha_0-\alpha_i$ or $2\alpha_0+\alpha_i-\alpha_1-\ldots-\alpha_5$ for some $1\leq i\leq 5$.
The classes of the conics have to add up to $-2\kappa$, and must be of the following form
\[
\scriptsize
(\alpha_0-\alpha_i)+(2\alpha_0+\alpha_i-\alpha_1-\ldots-\alpha_5)
+
(\alpha_0-\alpha_j)+(2\alpha_0+\alpha_j-\alpha_1-\ldots-\alpha_5),
\]
for some $1\leq i,j\leq 5$.
Since $(\alpha_0-\alpha_i)\cdot (2\alpha_0+\alpha_i-\alpha_1-\ldots-\alpha_5)=2$ 
there are two co-spherical circles.
\end{proof}

\begin{lem}
\label{DQ}
If $F, F'\subset D\times \P^1$ are families of circles on a Darboux cyclide $D\subset\S^3$ such that $F\cdot F'=1$ and $\o\in D$,
then there exists a unique doubly ruled quadric surface $Q\subset S$ such that $\e\in Q$, $Q\notin \U_\o$, $\orb_\o(Q)=D$
and the two rulings of $Q$ correspond via $\orb_\o$ to $F$ and $F'$.
\end{lem}

\begin{proof}
Let $C,C'\subset D$ be two circles in $F$ and $F'$ respectively, such that $C\cap C'=\o$.
By Lemma~\ref{l} there exist unique lines $\ell,\ell'\subset S$ 
containing $\e$ such that $\orb_\o(\ell)=C$ and $\orb_\o(\ell')=C'$.
We choose some point on $h\in \ell'$ and let 
$C''$ be the unique circle in the family $F$ that passes through $p:=\orb_\o(h)$.
We apply Lemma~\ref{l} with $\o$ replaced by $p$
and obtain a unique line $L\subset S$ containing $\e$.
It follows from the construction that $C''=\orb_\o(\ell'')$, 
where $\ell'':=hL$ and $hL$ means the image of each point in $L$
multiplied with the dual quaternion $h$.
\begin{center}
\fig{2}{8}{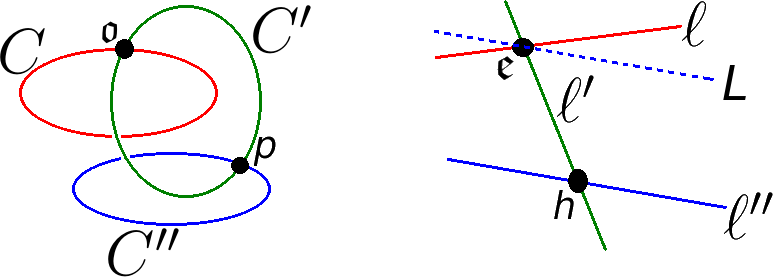} 
\end{center}
Thus we obtain three intersecting lines $\ell,\ell',\ell''\subset S$ spanning a 3-space $V$ and 
three circles $C,C',C''\subset D$ that pairwise intersect in at most one point.
We have that $V\nsubseteq S$, otherwise $\orb_\o(V)$ would be either a point or a 2-sphere by Lemma~\ref{l}.
Thus $V\cap S$ defines a unique quadric surface $Q\subset S$ 
such that $\ell,\ell',\ell''\subset Q$ and $Q\notin \U_\o$.
It follows from Lemma~\ref{pmz} that $D':=\orb_\o(Q)$ is a Darboux cyclide 
such that $C,C',C''\subset D\cap D'$. 
It follows from Lemma~\ref{darboux} that $D$ and $D'$ must be equal.
By Lemma~\ref{l}, lines in $S$ correspond to circles in $\S^3$
and a quadric $Q\subset S$ is covered by two families of lines.
This concludes the proof of this lemma. 
\end{proof}

\begin{rem} \label{rimas}
For the existence statement in Lemma~\ref{DQ}, it is also possible to give an algebraic proof
without using Lemma~\ref{darboux}. By \citep[Theorem~11]{j}, there exists a parametrization 
for Darboux cyclides of bidegree $(2,2)$ 
so that the parameter curves are the circles
in $F_1$ and $F_2$, respectively. 
Lifting the parametrization to $\S^3$, we obtain 5 biquadratic polynomials
$X_0,\dots,X_5\in\R[s,t]$ such that $X_0X_4=X_1^2+X_2^2+X_3^2$ and 
$(X_0:X_1:X_2:X_3:X_4)$ is a parametrization of the Darboux cyclide $D$.
By \citep[Theorem 3]{kz}, there exist bilinear polynomials $A,B\in\H[s,t]$ with quaternion coefficients such that
\[ 
\norm(A)=X_0,\qquad \norm(B)=X_4,\qquad AB = X_1\qi+X_2\qj+X_3\qk. 
\]
The bilinear polynomial $H := \bar{A}+\eps B\in\D\H[s,t]$ then defines a parametrization of a nonsingular ruled
quadric in the Study quadric $S$, and the image of this quadric via $\orb_\o$ is exactly $D$.
\Mend
\end{rem}

\begin{thm}
\label{main}
The map $\orb_\o\colon S\setminus F_\o\to \S^3$ defines a one-to-one correspondence between 
two families of lines that cover a quadric surface $Q\subset S$
such that $\e\in Q$ with $Q\notin \U_\o$ --- and --- two non-cospherical families of circles 
that cover a Darboux cyclide $D\subset \S^3$ such that $\o\in D$.
\end{thm}

\begin{proof}
The left to right direction is a consequence of Lemma~\ref{pmz} and Lemma~\ref{l}.
The converse direction follows from Lemma~\ref{DQ}. 
\end{proof}

For example, let $D\subset \S^3$ be a surface that is 
covered by exactly $6$ families of circles and suppose that $\o\in D$. 
In this case $D$ is a Darboux cyclide and admits $15$ pairs of families of circles \cite{blum,s3}. 
There are exactly three pairs $(F,F')$ of families of circles such that $F\cdot F'=2$ (see Figure~\ref{fig}). 
For the remaining 12 pairs $(F,F')$ of families one has $F\cdot F'=1$ 
and thus by Theorem~\ref{main} there are 12 quadric surfaces $Q\subset S$ such that $\e\in Q$ and $\orb_\o(Q)=D$.

\section{Acknowledgements}

We thank R. Krasauskas for the algebraic proof of the existence statement in Lemma~\ref{DQ}
(see Remark~\ref{rimas}) and M. Skopenkov for the help in the analysis of the algebraic statements
underlying this proof.

\bibliographystyle{plainnat}
\bibliography{paper}

\begin{thebibliography}{11}
\providecommand{\natexlab}[1]{#1}
\providecommand{\url}[1]{\texttt{#1}}
\expandafter\ifx\csname urlstyle\endcsname\relax
  \providecommand{\doi}[1]{doi: #1}\else
  \providecommand{\doi}{doi: \begingroup \urlstyle{rm}\Url}\fi

\bibitem[Blum(1980)]{blum}
R.~Blum.
\newblock Circles on surfaces in the {E}uclidean {$3$}-space.
\newblock In \emph{Geometry and differential geometry ({P}roc. {C}onf., {U}niv.
  {H}aifa, {H}aifa, 1979)}, volume 792 of \emph{Lecture Notes in Math.}, pages
  213--221. Springer, Berlin, 1980.

\bibitem[Dolgachev(2012)]{dol}
I.~V. Dolgachev.
\newblock \emph{Classical algebraic geometry: A modern view}.
\newblock Cambridge University Press, Cambridge, 2012.

\bibitem[Husty and Schr{\"o}cker(2010)]{hs}
M.~Husty and H.-P. Schr{\"o}cker.
\newblock Algebraic geometry and kinematics.
\newblock In \emph{Nonlinear computational geometry}, volume 151 of \emph{IMA
  Vol. Math. Appl.}, pages 85--107. Springer, 2010.

\bibitem[J\"uttler(1993)]{bj}
B.~J\"uttler.
\newblock {\"U}ber zwangl\"aufige rationale {B}ewegungsvorg\"ange.
\newblock \emph{\"Osterreich. Akad. Wiss. Math.-Natur. Kl. Sitzungsber. II},
  202\penalty0 (1-10):\penalty0 117--132, 1993.

\bibitem[Krasauskas and Skopenkov(2015)]{s1}
R.~Krasauskas and M.~Skopenkov.
\newblock Surfaces containing two circles through each point and {P}ythagorean
  6-tuples.
\newblock \emph{arXiv:1503.06481v2}, 2015.

\bibitem[Krasauskas and Zube(2014)]{kz}
R.~Krasauskas and S.~Zube.
\newblock Rational bezier formulas with quaternion and clifford algebra
  weights.
\newblock \emph{SAGA - Advances in ShApes, Geometry, and Algebra, Geometry and
  Computing}, 10:\penalty0 147--166, 2014.

\bibitem[Li et~al.(2016)Li, Schicho, and Schr\"ocker]{zjs}
Z.~Li, J.~Schicho, and H.-P. Schr\"ocker.
\newblock The rational motion of minimal dual quaternion degree with prescribed
  trajectory.
\newblock \emph{Comput. Aided Geom. Design}, 41:\penalty0 1--9, 2016.

\bibitem[Peternell(2012)]{p}
M.~Peternell.
\newblock Generalized dupin cyclides with rational lines of curvature.
\newblock \emph{Lecture Notes in Computer Science}, 6920:\penalty0 543--552,
  2012.

\bibitem[Pottmann et~al.(2012)Pottmann, Shi, and Skopenkov]{s3}
H.~Pottmann, L.~Shi, and M.~Skopenkov.
\newblock Darboux cyclides and webs from circles.
\newblock \emph{Comput. Aided Geom. Design}, 29\penalty0 (1):\penalty0 77--97,
  2012.

\bibitem[R\"oschel(1985)]{or}
O.~R\"oschel.
\newblock Rationale r\"aumliche {Z}wangl\"aufe vierter {O}rdnung.
\newblock \emph{\"Osterreich. Akad. Wiss. Math.-Natur. Kl. Sitzungsber. II},
  194\penalty0 (4-10):\penalty0 185--202, 1985.

\bibitem[Schicho(2001)]{j}
J.~Schicho.
\newblock The multiple conical surfaces.
\newblock \emph{Beitr. Alg. Geom.}, 42:\penalty0 71--87, 2001.

\end{thebibliography}

\vspace{-3mm}
\paragraph{}
N. Lubbes,
Johann Radon Institute for Computational and Applied 
Mathematics (RICAM), Austrian Academy of Sciences
\\
\textbf{email:} niels.lubbes@gmail.com

\vspace{-3mm}
\paragraph{}
J. Schicho, 
Research Institute for Symbolic Computation (RISC), 
Johannes Kepler University
\\
\textbf{email:} josef.schicho@risc.jku.at
\end{document}